\documentclass[11pt, reqno]{amsart}
\usepackage{amssymb,verbatim,amscd,amsmath}

\usepackage{braket}
\usepackage{array}
\usepackage{mathtools}
\usepackage{verbatim}
\usepackage{caption}
\usepackage{subcaption}
\usepackage{pdfsync}
\usepackage{tabu}
\usepackage{fullpage}
\usepackage{color}
\usepackage[svgnames]{xcolor}
\usepackage{enumerate}
\usepackage{pst-node,pst-text,pst-3d,pstricks}
\usepackage[dvips,all]{xy}
\usepackage{hyperref}
\usepackage{graphics}
\usepackage{tikz}
\usepackage{pst-plot}
\usepackage{cleveref}

\usepackage{tikz-cd}

\usetikzlibrary{patterns}
\usetikzlibrary{calc}
\usetikzlibrary{matrix}
\usetikzlibrary{arrows}
\usetikzlibrary{decorations.pathmorphing}
\usepackage{graphicx,supertabular,paralist}
\usetikzlibrary{decorations.pathreplacing, decorations.markings}
\tikzset{->-/.style={decoration={markings,mark=at position #1 with {\arrow{>}}},postaction={decorate}}}
\usepgflibrary{arrows}

\numberwithin{equation}{section}
\newtheorem{theorem}{Theorem}[section]
\newtheorem{lemma}[theorem]{Lemma}
\newtheorem{proposition}[theorem]{Proposition}
\newtheorem{corollary}[theorem]{Corollary}

\theoremstyle{definition}
\newtheorem{definition}[theorem]{Definition}
\newtheorem{fact}[theorem]{Fact}

\newtheorem{def-prop}[theorem]{Definition-Proposition}
\newtheorem{remark}[theorem]{Remark}
\newtheorem{example}[theorem]{Example}

\newtheorem{notation}[theorem]{Notation}
\newtheorem{question}[theorem]{Question}

\newtheorem*{Mysketch}{Sketch of proof} 
  {\pushQED{\qed}\begin{Mysketch}}
  {\popQED\end{Mysketch}}
  
  \newtheorem*{Myproof}{Proof of the Claim} 
  {\pushQED{\qed}\begin{Myproof}}
  {\popQED\end{Myproof}}

\DeclareMathOperator{\pol}{pol}
\DeclareMathOperator{\Tor}{Tor}
\DeclareMathOperator{\reg}{reg}
\DeclareMathOperator{\Ann}{Ann}
\DeclareMathOperator{\Ass}{Ass}

\DeclareMathOperator{\height}{ht}

\DeclareMathOperator{\link}{link}

\DeclareMathOperator{\depth}{depth}
\DeclareMathOperator{\pd}{pd}

\DeclareMathOperator{\st}{st}

\DeclareMathOperator{\susp}{susp}

\newcommand{\G}{\mathcal{G}}

\newcommand{\ZZ}{{\mathbb Z}}
\newcommand{\NN}{{\mathbb N}}

\def\pp{{\frak p}}

\def\G{{\mathcal G}}

\def\G{{\mathcal G}}

\def\a{{\bf a}}

\def\1{{\bf 1}}
\def\0{{\bf 0}}

\begin{document}

\title{On algebraic and combinatorial properties of weighted simplicial complexes}

\author{Selvi Kara}
\address{Department of Mathematics and Statistics, University of South Alabama,  411 University Blvd North, Mobile, AL 36688-0002, USA}
\email{selvi@southalabama.edu}
\urladdr{}

\begin{abstract} 
Weighted simplicial complexes (WSCs) are powerful tools for describing weighted cloud data or networks with weighted nodes. In this paper, we propose a novel approach to study WSCs via the concept of polarization. Polarization of a WSC allows one to construct a new (unweighted) simplicial complex which coincides with an object called the mixed wreath product. This new construction preserves several properties and invariants of the underlying simplicial complex of a WSC. Our main focus is to analyze WSCs through their underlying simplicial complexes and mixed wreath products. Combinatorially, we investigate properties such as vertex-decomposability, shellability, constructibility; algebraically, we study Betti numbers, associated primes and primary decompositions of ideals associated to WSCs.
 
\end{abstract}


\maketitle


\section{Introduction}


Simplicial complexes are versatile objects due to their combinatorial nature and used in many fields of mathematics such as commutative algebra and topology. They are also important tools for describing network structures such as protein networks \cite{ESTRADA201846},  brain networks \cite{giusti2016two}, collaboration networks \cite{patania2017shape, ramanathan2011beyond}, and many other structures in which there are interactions between nodes.  In applications, however, simplicial complexes are often weighted: for instance, in scientific collaboration networks, teams of collaborators can be weighted by a measure of the strength of their collaboration (e.g.,  \cite{courtney2017weighted}).  Recently, weighted simplicial complexes have become popular in topological data analysis (TDA)  to study weighted cloud data by replacing the data with a family of weighted simplicial complexes (see \cite{courtney2017weighted, giusti2016two, sharma2017weighted}).




The central purpose of this paper is to study weighted simplicial complexes (WSCs)  through associated  combinatorial and algebraic objects. We construct a new (unweighted) simplicial complex from a WSC via polarization which keeps many properties and invariants of the underlying simplicial complex of a WSC unchanged.  In particular, we
\begin{itemize}
\item consider \emph{mixed wreath products} of simplicial complexes and derive some of their properties  (\Cref{mixedwreath}),
\item  introduce the concept of polarization for WSCs;  polarization of a WSC coincides with the mixed wreath product (\Cref{Polar}), and
\item  study monomial ideals associated to WSCs from a broader perspective through  \emph{weighting} operation of monomial ideals  (\Cref{weightedmonomial}).
\end{itemize}

There are several different ways to define a WSC in the literature (see  \cite{dawson1990homology},\cite{giusti2016two},  \cite{kannan2019persistent}) and we adopt the following definition in this paper.


\begin{definition}\label{def:WSC}
Let $\Delta$ be a simpicial complex on the vertex set $V$ and let $w: V \rightarrow \NN$ be a function called the \emph{weight function} which assigns a weight to each vertex of $\Delta.$  A \emph{weighted simpicial complex} is a pair $(\Delta,w)$ consisting of a simplicial complex $\Delta$ and a weight function $w$ on the set of vertices of $\Delta.$ We call $\Delta$ the \emph{underlying simplicial complex} of $(\Delta,w).$
\end{definition}


WSCs were first introduced in 1990 by Dawson \cite{dawson1990homology} in which a weight function is defined on a simplicial complex together with a divisibility condition for face weights. Dawson also established weighted simplicial homology which coincides with the simplicial homology when all the faces have the same weight. However, it differs from the simplicial homology in several ways, for instance, it is possible for the zeroth weighted homology group of a WSC to have torsion.  We refer the reader to \cite{dawson1990homology} for more details.


Apart from the use of WSCs as models in TDA, there is not much known about the combinatorial or algebraic nature of WSCs. One of the few results concerning WSCs is that the Stanley-Reisner ring of a WSC is isomorphic to the integral cohomology ring of a special polyhedral product called the weighted Davis-Januszkiewicz space (see \cite[Theorem 10.5]{bahri2015operations}). Our approach to study WSCs concerns their underlying simplicial complexes and the mixed wreath products. 

In \Cref{mixedwreath},  we focus on the construction of the mixed wreath product of a simplicial complex $\Delta$ and effects of this construction on the properties of $\Delta.$  The key idea behind the mixed wreath product is the one-point suspension introduced by Provan and Billera in \cite{provan1980decompositions}; the first generalization of one-point suspensions called  \emph{wreath products}  were studied in \cite{joswig2005one}. The main results of this section given below generalize results from \cite{joswig2005one}. In particular, we provide formulas for the $f$-vectors of the mixed wreath product in terms of the underlying simplicial complex.

(\textbf{\Cref{prop:fvector}}) \emph{Given a $d$-dimensional simplicial complex $\Delta$ on $n$ vertices and positive integers $ d_1, \ldots, d_n$,  the $f$-vector of the mixed wreath product $\partial \Delta_{(d_1,\ldots,d_n)} \wr \Delta$ can be expressed in terms of $f$-vectors of $\Delta.$  In particular, we have
$$f_0 (\partial \Delta_{(d_1,\ldots,d_n)} \wr \Delta) = \sum_{i=1}^n d_i +f_0(\Delta), \text{ and } $$
$$\displaystyle f_{\tilde{d} }  (\partial \Delta_{(d_1,\ldots,d_n)} \wr \Delta)  = \sum_{i=1}^{f_{d} (\Delta)} \prod_{j=1}^{n-d-1} (d_{i,j}+1)$$
where  $\displaystyle \tilde{d} := \dim ( \partial \Delta_{(d_1,\ldots,d_n)} \wr \Delta ) = \sum_{i=1}^n d_i+ d$  and complement of $F_i$ is a set $ \{ v_{i,1}, \ldots, v_{i,n-d-1}\}$ for each $d$-dimensional facet $F_i$ for $1\leq i \leq f_{d} (\Delta) .$}

In addition to the above proposition, we investigate the properties preserved under the mixed wreath product construction.

(\textbf{\Cref{cor:properties}}) \emph{Let $\Delta$ be a simplicial complex on $n$ vertices and $ d_1, \ldots, d_n$ be  positive integers. Suppose $\mathcal{P}$ is one of the following properties:
\begin{itemize}
\item vertex-decomposable,
\item shellable, 
\item constructible.
\end{itemize}
 The mixed wreath product  $\partial \Delta_{(d_1,\ldots,d_n)} \wr \Delta$ of $\Delta$ is $\mathcal{P}$ if and only $\Delta$  is $\mathcal{P}.$ }



A central notion introduced in this paper is the  concept of polarization defined in \Cref{Polar}.  Polarization is an operation applied to a WSC that results in  the mixed wreath product of the underlying simplicial complex studied in \Cref{mixedwreath}. The key idea behind polarization is to view each weighted vertex of  a WSC as a full simplex. During the polarization process, we replace a vertex of weight $d$  with a full simplex on $d$-vertices while connecting  boundaries of each full simplex to the rest of the ``polarized" simplicial complex. 
%

Polarization of a WSC is inspired by the monomial polarization operation and can be seen as its combinatorial analogue.  As in the Stanley-Reisner theory, one can associate each WSC with a monomial ideal called the \emph{Stanley-Reisner ideal}. Thus, two polarizations are naturally related via Stanley-Reisner ideals (\Cref{cor:pol1}).  Through this connection, we obtain a list of important shared properties between Stanley-Reisner ideals of a WSC and its polarization (\Cref{cor:pol}).



The results from \Cref{mixedwreath}  and \Cref{Polar} provide a  glimpse of the combinatorial nature of WSCs. In \Cref{weightedmonomial}, we turn our attention towards the algebraic investigation of WSCs through Stanley-Reisner ideals.

Stanley-Reisner ideals of a WSC $(\Delta,w)$ and its underlying complex $\Delta$ are related through a monomial operation called \emph{weighting}. More specifically, the Stanley-Reisner ideal of $(\Delta,w)$ is a \emph{weighted ideal} of the Stanley-Reisner ideal of $\Delta.$  In  \Cref{weightedmonomial},  we follow a more general treatment of the weighting operation by focusing on monomial ideals (i.e., not only Stanley-Reisner ideals of simplicial complexes)  and their weighted ideals. A main result of this section,  presented below, relates the Betti numbers of a monomial ideal and its weighted ideal. 

(\textbf{\Cref{cor:betti}}) \emph{Let $I \subseteq R=k[x_1,\ldots, x_n]$ be a monomial ideal and $(I,w)$ be its weighted ideal in $R$ where $w_i:= w(x_i)$ for $1 \leq i \leq n.$ Then
$$\beta_{i, (j_1, \ldots, j_n)} (I) = \beta_{i,(w_1j_1, \ldots, w_n j_n)} ((I,w))$$
 for each $(j_1, \ldots, j_n) \in \mathbb{N}^n.$ }


Betti numbers contain important information related to the structure and shape of  many objects from algebra, topology and geometry. For example, in TDA, one can encode the persistent homology of the data set in a ``barcode" using Betti numbers (see \cite{edelsbrunner2008persistent, ghrist2008barcodes}). In the case of weighted cloud data, current computations of weighted persistent homology (\cite{meng2020weighted, ren2018weighted}) are based on the homology theory developed in  \cite{dawson1990homology}, which results in vanishing homologies for certain WSCs regardless of the topological structure of the underlying simplicial complex (\cite[Theorem 3.1]{dawson1990homology}). In contrast, \Cref{cor:betti} provides all Betti numbers of a WSC without introducing new homology theories; moreover, it preserves the topological information of the underlying simplicial complex otherwise lost by methods of  \cite{dawson1990homology}.  


As a natural consequence of  \Cref{cor:betti}, we establish equality of certain algebraic invariants and properties of a monomial ideal and its weighted ideal. In addition, by using completely different techniques, we  provide more compact proofs to several results from \cite{sayedsadeghi2018normally} concerning associated primes, irredundant primary decompositions, normal torsion freeness, and the weighting operation.

Our paper is organized as follows. In \Cref{ch:pre}, we collect the relevant terminology that will be used in the paper. In \Cref{mixedwreath}, we study mixed wreath products of simplicial complexes. In \Cref{Polar}, we introduce the polarization of a weighted simplicial complex and obtain several relations between the Stanley-Reisner ideals of a WSC and its polarization. In \Cref{weightedmonomial}, we examine the effects of the weighting operation on algebraic invariants and the structure of monomial ideals.

\section{Preliminaries}\label{ch:pre}
In this section, we collect the necessary notation and terminology from Stanley-Reisner theory and commutative algebra. 

\subsection{Stanley-Reisner theory.} A  \emph{simplicial complex} $\Delta$ on a (finite) vertex set $V= V(\Delta)$ is a collection of subsets $F \subseteq V(\Delta)$ called  \emph{faces}  (or simplices) with the property that if $F \in \Delta $ and $G \subseteq F,$ then $G \in \Delta.$  The \emph{dimension}  of a face is $\dim (F) =|F| -1$ and the dimension of $\Delta$ is $\dim (\Delta) = \max \{ \dim (F) ~|~ F \in \Delta \}.$ A \emph{facet} of $\Delta$ is a maximal face under inclusion, and we say that $\Delta$ is \emph{pure} if all of its facets have the same dimension.

A \emph{minimal non-face} of $\Delta$ is a set $M \subseteq V$ with $M \notin \Delta,$ but $M\setminus i \in \Delta$ for every $i \in M.$ Note that a face of $\Delta$ may be described as a set not containing a minimal non-face; hence, $\Delta$ is completely determined by its minimal non-faces.  We denote the set of all minimal non-faces of $\Delta$ by $\mathcal{M}(\Delta).$

Let $R=k[x_1,\ldots, x_n]$ be a polynomial ring.  The \emph{Stanley-Reisner ideal} of a simplicial complex $\Delta$ on a vertex set $V(\Delta) =  [n]=\{1,\ldots, n\}$  is an ideal in $R$ defined as
$$I_{\Delta} = \Big( \prod_{i \in F} x_i ~:~ F \in \mathcal{M}(\Delta) \Big)$$
where vertex $i$ of $\Delta$ is identified with variable $x_i$ in $R$ for each $1\leq i \leq n.$ The quotient ring $R/I_{\Delta}$ is called the \emph{Stanley-Reisner ring} of $\Delta$ and denoted by $k[\Delta].$ 


The \emph{full $d$-simplex},  denoted by $\Delta_{d},$ is the simplicial complex consisting of all subsets of $[d+1]$ and its dimension is equal to $d.$ Its boundary $\partial \Delta_{d}$ is the simplicial complex on $[d+1]$ consisting of all proper subsets of $[d+1]$ and the dimension  of the boundary is equal to $d-1.$

The \emph{join} of two simplicial complexes $\Delta$ and $\Gamma$ on the disjoint set of vertices is 
$$\Delta \ast \Gamma := \{ \sigma \cup \tau ~|~ \sigma \in \Delta, \tau \in \Gamma\}.$$
Note that the set of minimal non-faces of $\Delta \ast \Gamma$ is the union of the sets of minimal non-faces of $\Delta$ and $\Gamma.$ For a simplicial complex $\Delta$, a \emph{cone} of $\Delta$ is a join of $\Delta$ with a single vertex; a \emph{suspension} of $
\Delta$ is denoted by $\susp (\Delta)$ and defined as $\susp (\Delta) = \partial \Delta_1 \ast \Delta  $ where  $\Delta_1$  denotes a full 1-simplex.

Let $\Delta$  be simplicial complex and $v$ be a vertex in $\Delta.$  The \emph{star}  of $v$ in $\Delta$ and the \emph{link} of $v$ in $\Delta$ both describe the local structure of $\Delta$ around $v.$  
$$\st_{\Delta} (v) := \{ G\in \Delta ~|~ G \cup \{v\} \in \Delta  \}$$
$$\link_{\Delta} (v) := \{ G \in \st_{\Delta} (v)  ~|~ G \cup \{v\} = \emptyset  \}$$

Note that $\st_{\Delta} (v) = \{ v\} \ast \link_{\Delta}(v).$ The \emph{deletion} of $F \subseteq V(\Delta)$ from $\Delta$ is defined as
$$\Delta \setminus F = \{ G \in \Delta ~|~ F  \nsubseteq G\}.$$

\subsection{Commutative algebra}

Let $k$ be a field, let $R = k[x_1, \ldots, x_n]$ be a standard graded polynomial ring over $n$ variables.

\begin{definition}
Let $M$ be a finitely generated graded $R$-module  and let 

 $$
 0\longrightarrow \bigoplus_{j \in \ZZ} R(-j)^{\beta_{p,j}(M)} \longrightarrow \bigoplus_{j \in \ZZ} R(-j)^{\beta_{p-1,j}(M)} \cdots \longrightarrow 
\bigoplus_{j \in \ZZ} R(-j)^{\beta_{0,j}(M)} \longrightarrow M \longrightarrow 0$$ 

be its minimal free resolution. The symbol $\beta_{i,j} (M)$ denotes the number of copies of $R(-j)$ appearing in the $i$th step of the minimal free resolution of $M$ and we call $\beta_{i,j} (M)$  the \emph{graded Betti number} of $M.$ Graded Betti numbers of $M$ are also given by $\beta_{i,j} (M) = \dim_k \Tor_i (M,k)_j.$  There are several important invariants associated to Betti numbers of $M$  such as the regularity, the projective dimension and the depth. The \emph{Castelnuovo-Mumford regularity} (or, simply \emph{regularity}) of $M$ can be defined by
$$\reg (M) = \max \{ j-i : \beta_{i,j} (M) \neq 0\},$$
the \emph {projective dimension} of $M$ is defined by 
$$\pd (M) = \max \{ i : \beta_{i,j} (M) \neq 0\},$$
and the \emph{depth} of $M$ can be written as
$$\depth (M) = n - \pd (M)$$
by the Auslander-Buchsbaum formula.
\end{definition}

%

Below, we recall the definition of polarization of a monomial ideal (for more details, see \cite{peeva2010}).


\begin{definition}{ \cite[Construction 21.7]{peeva2010}}
Let $R=k[x_1,\ldots, x_n]$ be a polynomial ring over a field $k.$ Given a $n$-tuple $\a=(a_1, \ldots, a_n) \in \ZZ_{\ge 0}^n,$ let $\bf{x^a}$ denote the monomial $x_1^{a_1}\cdots x_n^{a_n} \in R.$ 
\begin{enumerate}
\item The polarization of $\bf{x^a}$ is defined to be $(x^{\a})^{\pol} ,$ where $(\bullet)^{\pol}$ replaces $x_i^{a_i}$ by a product of distinct variables $\prod_{j=1}^{a_i} x_{i,j}.$ \\
\item Let $I=(\bf{x^{a_1}}, \ldots, \bf{x^{a_r}}) \subseteq$$R$ be a monomial ideal. The \emph{polarization} of $I$ is defined to be the ideal $I^{\pol} = ((\bf{x^{a_1}})^{\pol}, \ldots, (\bf{x^{a_r}})^{\pol})$ in a new polynomial ring $R^{\pol} = k [x_{i,j} ~|~ 1 \le i \le n, 1 \le j \le p_i],$ where $p_i$ is the maximum power of $x_i$ appearing in $ \bf{x^{a_1}}, \ldots, \bf{x^{a_r}}.$
\end{enumerate}
\end{definition}

\section{Mixed Wreath Products of Simplicial Complexes}\label{mixedwreath}

\subsection{One-point suspensions and reduced joins} The building blocks of the mixed wreath product construction are the one-point suspension and reduced join operations. In this subsection, we consider these two operations along with their properties.

\begin{definition}
Let $\Delta$ be a simplicial complex and $v$ be a vertex in $\Delta.$ The \emph{one-point suspension}  of $\Delta$  with respect to $v$,  denoted by $\Delta(v),$ is the simplicial complex 
\begin{equation}\label{eq:wedge}
\Delta(v) := \Big( \{ \{v_1\}, \{v_2\} \} \ast (\Delta \setminus \{v\})  \Big) \cup  \Big( \{v_1,v_2\} \ast \link_{\Delta} (v)   \Big)
\end{equation}
where $v_1$ and $v_2$ are two copies of the vertex $v$ that are not contained in $\Delta.$ 
\end{definition}

This operation was first introduced in  \cite{provan1980decompositions} by Provan and Billera. Authors of  \cite{provan1980decompositions} called $\Delta(v)$ the simplicial wedge of $\Delta$ on $v$ and named the operation ``simplicial wedge." Throughout the paper, we use one-point suspension to refer to this operation.

\begin{figure}[!htb]
    \centering
    \begin{minipage}{.5\textwidth}
        \centering
\includegraphics[width=0.5\textwidth]{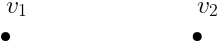}
\caption{Simplicial complex $\Delta$}
  \label{fig:simp}
    \end{minipage}%
    \begin{minipage}{0.5\textwidth}
        \centering
\includegraphics[width=0.5\textwidth]{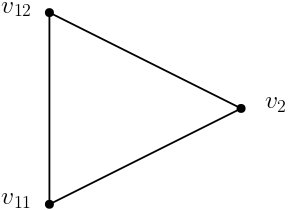}
\caption{$\Delta (v_1)$}
  \label{fig:one}
    \end{minipage}
\end{figure}

\begin{example}
Let $\Delta$ be the disjoint union of two points given in \Cref{fig:simp} with facets $\{v_1\}, \{v_2\}.$ One-point suspension $\Delta (v_1)$ has three vertices $\{ v_{11}, v_{12},v_2\}$. Since $\link_{\Delta} (v_1) = \emptyset,$ and $\Delta \setminus \{v_1\} = \{v_2\}, $  by (\ref{eq:wedge}),
$$\Delta (v_1) =  \Big( \{ \{v_{11} \}, \{ v_{12} \} \} \ast \{v_2\} \Big) \cup  \Big( \{ v_{11},v_{12} \} \ast \emptyset  \Big)  =  \{ v_{11}, v_{12}\} \cup \{ v_{11},v_2\} \cup \{v_{12},v_2 \}$$ 
which is the boundary of a two-simplex as given in \Cref{fig:one}. 
\end{example}

\begin{remark}\label{rem:facet}
The facets of $\Delta(v)$ are closely related  to facets of $\Delta$ and they depend whether a facet of $\Delta$ contains $v$ or not. 
Let $F$ be a facet of $\Delta,$ then
\begin{itemize}
\item  if $F$ does not contain $v,$  we have $F \cup \{v_1\}$ and $F \cup \{v_2\}$  as facets of $\Delta(v),$ and, 
\item  if $v \in F,$ then $( F \setminus \{v\} ) \cup \{v_1, v_2\}$ is a facet of $\Delta(v).$
  \end{itemize}
  
 Every minimal non-face of $\Delta(v)$ containing $v_1$ also contains $v_2$ and vice versa. Furthermore, a minimal non-face $\{v_{i_1}, \ldots, v_{i_k}\}$ of $\Delta$ with $v_{i_j} \neq v$ remains a minimal non-face of $\Delta(v).$ One can reverse the one-point suspension operation to obtain $\Delta$ from $\Delta(v),$ more specifically, simplicial complex $\Delta$ is the link of $v_2$  in $\Delta(v)$ by setting $v_1=v.$
\end{remark}

\begin{remark}\label{rem:suspension}
There is an equivalent definition for  one-point suspension and structure of this definition exploits the combinatorial relation between $\Delta$ and $\Delta(v).$  In particular, we observe that standard suspension and one-point suspension of a simplicial complex $\Delta$ are combinatorially related. 

Let $\{v_1,v_2\}$ be the copies of $v$ in $\Delta(v)$ and let $\Delta_1$ be the full simplex on vertices $\{v_1,v_2\}$ with $\partial \Delta_1= \{\{v_1\}, \{v_2\}\}.$  
\begin{equation}\label{eq:one-point}
\Delta(v) =  \Big(   ( \partial  \Delta_1   \ast \Delta  )  \setminus  (\partial \Delta_1  \ast \st_{\Delta} (v) )  \Big) \cup  \Big(  \Delta_1 \ast \link_{\Delta} (v)\Big)
\end{equation}
Note that $\partial  \Delta_1  \ast \Delta $ is the standard suspension of $\Delta.$ It follows from \Cref{eq:one-point} that $ \Delta(v)$ is obtained from $\partial  \Delta_1  \ast \Delta $ by a bistellar flip which removes $\partial  \Delta_1  \ast \st_{\Delta} (v)  $ from $\partial \Delta_1  \ast \Delta  $ and inserts  $ \Delta_1  \ast \link_{\Delta} (v)$ instead. (see \cite{izmestiev2017simplicial} for a definition and further references on bistellar flips). Then  $\susp (\Delta)$ and $\Delta(v)$ are PL-homeomorphic by \cite{pachner1991pl} (see \cite{hudson1967piecewise} for the definition of a PL-homeomorphism and more details in PL-topology).
\end{remark}

One-point suspension is an operation that can be applied iteratively to a simplicial complex and its one-point suspensions. In particular, one point-suspension is a commutative operation when it is applied with respect to two different vertices of $\Delta.$

\begin{proposition}\label{prop:comm}
Let $\Delta$ be a simplicial complex with at least two (distinct) vertices $v_1$ and $v_2.$ The one-point suspension is a commutative operation, i.e., 
$$(\Delta(v_1))(v_2)=(\Delta(v_2))(v_1).$$
\end{proposition}

\begin{proof}
Let $\{v_1,v_2, \ldots, v_n\}$ be the set of vertices of $\Delta$ where $ V ( \Delta(v_1)) = \{v_{11},v_{12}, v_2, \ldots, v_n\}$ and $ V((\Delta(v_1)) (v_2) )= \{v_{11},v_{12}, v_{21},v_{22}, v_3, \ldots, v_n\}.$ It follows from  \Cref{rem:facet} that, for all facets $F$ of $\Delta,$ facets of $\Delta(v_1)$ are of either $ (F \setminus \{v_1\}) \cup \{v_{11}, v_{12}\} $ if $v_1 \in F,$ or $ F \cup \{v_{11} \} $ and $F\cup \{v_{12}\}$ if $ v_1 \notin F.$

Similarly, facets of $(\Delta(v_1))(v_2)$ can be determined transitively via facets $F$ of $\Delta.$ In particular, all facets of $(\Delta(v_1))(v_2)$  are in one of following forms.
\begin{itemize}
\item $(F \setminus \{v_1,v_2\} )\cup \{v_{11}, v_{12}\} \cup \{v_{21}, v_{22}\} $  if $v_1,v_2 \in F,$ 
\item  $ (F\setminus \{v_2\}  )\cup \{v_{11} \} \cup \{v_{21}, v_{22}\}$ or $(F \setminus \{v_2\}) \cup \{v_{12}\} \cup \{v_{21}, v_{22}\} $ if $ v_1 \notin F$ but $v_2 \in F,$
\item  $ (F \setminus \{v_1\}) \cup \{v_{11}, v_{12}\} \cup \{v_{21}\}$ or $(F \setminus \{v_1\}) \cup \{v_{11}, v_{12}\} \cup \{v_{22}\}$ if $v_1 \in F$ but $v_2 \notin F,$ 
\item $ F \cup \{v_{11} \} \cup \{v_{21} \},  F \cup \{v_{11} \} \cup \{v_{22} \},  F\cup \{v_{12}\} \cup \{ v_{21}\}, $ or $ F\cup \{v_{12}\} \cup \{ v_{22}\}$ if $v_1,v_2 \notin F.$ 
\end{itemize}
The roles of $v_1$ and $v_2$ can be exchanged and the result follows.
\end{proof}
 
Next, we consider the case when the iterative one-point suspension is applied with respect to a vertex and its copies in the one-point suspension.

\begin{remark}\label{rem:wedge} 
Let $\Delta$ be a simplicial complex with vertex $v$ and let $\{v_1,v_2\}$ be copies of $v$ in $\Delta(v).$ One-point suspension of $\Delta (v)$ with respect to $v_1$ or $v_2$ produces the same simplicial complex. 

Consider the simplicial complexes $(\Delta(v))(v_1)$  and  $(\Delta(v))(v_2)$  which are one-point suspensions of  $\Delta(v)$ with respect to $v_1$ and $v_2,$ respectively. Let 
 $\{v_{1,1}, v_{1,2}\}$ be copies of $v_1$  in $(\Delta(v))(v_1)$ and  $\{v_{2,1}, v_{2,2}\}$ be copies of $v_2$  in $(\Delta(v))(v_2).$ It is immediate that 
$$(\Delta(v))(v_1)=(\Delta(v))(v_2)$$
by identifying $v_{1,i} = v_{2,i}$ for $i=1,2.$  

In addition, by renaming $v$ and its copies  in $\Delta(v) (v_1)$  as $v_{11}, v_{12},v_{13},$ we have
\begin{eqnarray*}
 (\Delta(v))(v_1)  &=& \Big( \{ \{v_{11}, v_{12} \}, \{v_{11}, v_{13} \}, \{ v_{12}, v_{13}\}  \} \ast  (\Delta \setminus \{v\} )  \Big) \cup \Big( \{ v_{11},v_{12},v_{13} \} \ast \link_{\Delta} (v) \Big) \\
 &=& \Big(  ( \partial \Delta_{2}  \ast \Delta ) \setminus ( \partial  \Delta_{2} \ast \st_{\Delta} (v) ) \Big)  \cup \Big( \Delta_{2}  \ast \link_{\Delta} (v) \Big).
 \end{eqnarray*}
\end{remark}

Therefore, by comparing the above equality with \Cref{eq:one-point}, one can see that a higher dimensional analog of one-point suspension makes sense and it can be considered as  iterated one-point suspensions of $\Delta$ with respect to $v$ and copies of $v$ that are generated in each intermediate step (regardless of which copy is chosen). In  \cite{joswig2005one}, authors introduced the \emph{reduced join} operation to address such iterated one-point suspensions and we recall it below.

\begin{definition}{\cite[Definition 3.4]{joswig2005one}}
Let $\Delta$ be a simplicial complex and $v$ be a vertex of $\Delta.$ The \emph{reduced join} of $\Delta$ with $\partial \Delta_{d} ,$ the boundary  of a full $d$-simplex $\Delta_{d},$ with respect to $v$ is the simplicial complex 
$$ \partial \Delta_{d} \ast_v \Delta = \Big(  ( \partial \Delta_{d}  \ast \Delta ) \setminus ( \partial  \Delta_{d} \ast \st_{\Delta} (v)) \Big)  \cup \Big(  \Delta_{d}  \ast \link_{\Delta} (v)  \Big). $$
Note that this operation is equivalent to $d$-iterated one-point suspension of $\Delta$ with respect to $v$ and its copies from intermediate steps.
It follows from \Cref{rem:facet} that facets of $ \partial \Delta_{d} \ast_v \Delta $ are of the following form
$$ (F \setminus \{v\}) \cup \bigcup_{v \in F} \{ v_0, v_1, \ldots, v_{d} \} \cup \bigcup_{v\notin F } \{  v_0, v_1 \ldots, \widehat{v_i}, \ldots, v_{d}\}$$
for each facet $F$ of $\Delta.$ Practically, reduced join operation  (or $d$- iterated application of one-point suspension to $v$) replaces vertex $v$ with a full $d$-simplex and connects its boundary  to $\Delta \setminus \{v\}.$
\end{definition}

Similar to the discussion in \Cref{rem:suspension}, the reduced join $ \partial \Delta_{d} \ast_v \Delta $ is obtained from join $ \partial \Delta_{d}  \ast \Delta$ by a generalized bistellar flip and both simplicial complexes are PL-homeomorphic.

As an immediate consequence of \Cref{prop:comm}, the reduced join is a commutative operation because it is an iterated application of one-point suspension with respect to a vertex and its copies from each iteration.
\begin{corollary}{\cite[Proposition 3.8]{joswig2005one}}
Let $\Delta$ be a simplicial complex with at least two (distinct) vertices $v_1$ and $v_2.$ The reduced join operation is a commutative operation, i.e., 
$$  \partial \Delta_{d_1} \ast_{v_1} (  \partial \Delta_{d_2} \ast_{v_2}  \Delta  )  =   \partial \Delta_{d_2} \ast_{v_2} (  \partial \Delta_{d_1} \ast_{v_1}  \Delta  )  $$
for $d_1, d_2 \geq 1.$ 
\end{corollary}

\subsection{Mixed wreath products of simplicial complexes}

The mixed wreath product of a simplicial complex generates a highly symmetric simplicial complex from a given one. Resulting simplicial complex preserves various interesting combinatorial and topological properties of the original simplicial complex. In this subsection, we investigate several of these properties such as vertex decomposability, shellability, constructibility, and Cohen-Macaulayness by following the work of Joswig and Lutz \cite{joswig2005one}.  In the literature, mixed wreath products were also studied in \cite{bahri2015operations} from a toric topology point of view.

\begin{definition}\label{def:mixedwreath}
Let $\Delta$ be a simplicial complex with $n$ vertices, say $V(\Delta) = \{v_1,\ldots, v_n\},$ and let $(d_1, d_2, \ldots, d_n)$ be a sequence of positive integers. Let $\partial \Delta_{d_i}$ denote the boundary of a full $d_i$-simplex $\Delta_{d_i}$ for each $i.$  The \emph{mixed wreath product} $\partial \Delta_{(d_1, \ldots, d_n)} \wr \Delta$ of $\Delta$ is defined as follows.   

The vertices of $\partial \Delta_{(d_1, \ldots, d_n)} \wr \Delta$ are obtained by taking $(d_i+1)$ copies, say $v_{i,0}, v_{i,1}, \ldots, v_{i,d_i},$ of vertex $v_i$  for  each $1\leq i \leq n.$  
$$V (\partial \Delta_{(d_1, \ldots, d_n)} \wr \Delta) = \bigcup_{i=1}^n \{ v_{i,0}, \ldots, v_{i,d_i} \}$$
The facets of $ \partial \Delta_{(d_1, \ldots, d_n)} \wr \Delta$ arise from facets of $\Delta.$ In particular, facets of $\partial \Delta_{(d_1, \ldots, d_n)} \wr \Delta$ are all those subsets of $V(\partial \Delta_{(d_1, \ldots, d_n)} \wr \Delta) $ of the form
$$\displaystyle  \mathcal{F}= \bigcup_{v_i \in F} \{ v_{i,0},  \ldots, v_{i,d_i} \} \cup \bigcup_{v_j\notin F} \{ v_{j,0}, \ldots, \widehat{v_{j,k}}, \ldots, v_{j,d_j} \},$$
where $F$ is a facet of $\Delta.$ In the above expression $\{ v_{j,0},  \ldots, \widehat{v_{j,k}}, \ldots, v_{j,d_j} \}$, one of the vertices $\{ v_{j,0}, \ldots, v_{j,d_j} \}$ is omitted for the vertices $ v_j \notin F. $  Note that $  \dim (\mathcal{F})= \sum_{i=1}^n d_i+ \dim (F).$ 
\end{definition}

\begin{remark}
The dimension of the wreath product is $\displaystyle \sum_{i=1}^n d_i+ \dim (\Delta).$  Thus, the wreath product $\partial \Delta_{(d_1, \ldots, d_n)} \wr \Delta$ is pure if and only if $\Delta$ is pure.
\end{remark}

\begin{example}
\begin{figure}[h]
    \centering
    \begin{minipage}{.5\textwidth}
        \centering
\includegraphics[width=0.8\textwidth]{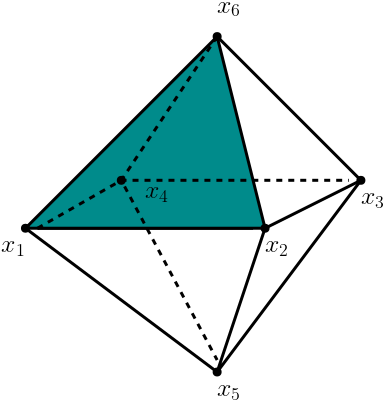}
\caption{ one facet of $\Delta$}
  \label{fig:empty}
    \end{minipage}%
    \begin{minipage}{0.5\textwidth}
        \centering
\includegraphics[width=0.8\textwidth]{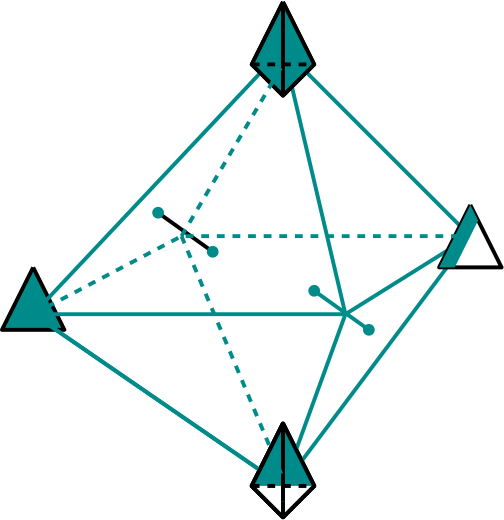}
\caption{one facet of $\partial \Delta_{(2,1,2,1,3,3)} \wr \Delta$}
  \label{fig:mixedwreath}
    \end{minipage}
\end{figure}
In Figure \ref{fig:mixedwreath}, we display one facet (in green) of the $14$-dimensional simplicial complex 
$\partial \Delta_{(2,1,2,1,3,3)} \wr \Delta$ where $\Delta$ is the octahedron given in Figure \ref{fig:empty}. This facet arises from the upper front triangle  $\{x_1,x_2,x_6\}$ of the octahedron.  Each vertex in the upper triangle of $\Delta$ contributes a full simplex $\Delta_2, \Delta_1,$ and $\Delta_3,$ in order, to a facet of the mixed wreath product and  each of the remaining vertices $x_3, x_4, x_5$ contribute  boundaries $\partial \Delta_2, \partial \Delta_1,$ and $\partial \Delta_3,$ in order. 
\end{example}

Note that the mixed wreath product is obtained from  $\Delta$ by successive reduced joins (in arbitrary order) with $\partial \Delta_{d_i}$ with respect to $v_i$ for all the vertices of $\Delta.$ Practically, one can obtain the mixed wreath product of $\Delta$ by replacing each vertex $v_i$ with a full  $d_i$-simplex and joining the boundary of the full simplex to the rest of the ``new" simplicial complex.

\begin{remark}
The \emph{wreath product}, defined in \cite[Definition 3.5]{joswig2005one} and denoted by $\partial \Delta_{d} \wr \Delta$, is a special case of the mixed wreath product. More specifically, $\partial \Delta_{(d_1, \ldots, d_n)} \wr \Delta = \partial \Delta_{d} \wr \Delta$ when $d_i =d$ for all $1\leq i \leq n.$ 
\end{remark}

\begin{remark}
If each $d_i=0$ for $1 \leq i \leq n, $ then $\partial \Delta_{(d_1, \ldots, d_n)} \wr \Delta  = \Delta$ and  $\partial \Delta_{(d_1, \ldots, d_n)} \wr \Delta = \Delta_{d_1}$  if $\Delta$ is a point. Also, if $\Delta$ is a full simplex on $n$ vertices, then  $\partial \Delta_{(d_1, \ldots, d_n)} \wr \Delta $ is a full simplex on $d_1 + \cdots +d_n+n$ vertices.
\end{remark}

Recall that the \emph{f-vector} of a $d$-dimensional simplicial complex $\Delta$ is the sequence  
$$f(\Delta) = (f_0, f_1, \ldots, f_d)$$
where $f_i$ is the number of $i$-dimensional faces of $\Delta$ for $0 \leq i \leq d.$

\begin{proposition}\label{prop:fvector}
Given a $d$-dimensional simplicial complex $\Delta$ on $n$ vertices, the $f$-vector of the mixed wreath product $\partial \Delta_{(d_1,\ldots,d_n)} \wr \Delta$ can be expressed in terms of $f$-vectors of $\Delta.$  In particular, we have
$$f_0 (\partial \Delta_{(d_1,\ldots,d_n)} \wr \Delta) = \sum_{i=1}^n d_i +f_0(\Delta), \text{ and } $$
$$\displaystyle f_{\tilde{d}~}  (\partial \Delta_{(d_1,\ldots,d_n)} \wr \Delta)  = \sum_{i=1}^{f_{d} (\Delta)} ~ \prod_{j=1}^{n-d-1} (d_{i,j}+1)$$
where $\tilde{d}:= \dim ( \partial \Delta_{(d_1,\ldots,d_n)} \wr \Delta)=\sum_{i=1}^n d_i+ d$ and complement of $F_i$ is a set $ \{ v_{i,1}, \ldots, v_{i,n-d-1}\}$ for each $d$-dimensional facet $F_i$ for $1\leq i \leq f_{d} (\Delta) .$
\end{proposition}

\begin{proof}
It follows from the definition that the vertex set of the mixed wreath product $\partial \Delta_{(d_1,\ldots,d_n)} \wr \Delta$  is formed by taking $d_i+1$ copies of vertex  $v_i$ of $\Delta$ for each $1\leq i \leq n.$ Hence, $f_0 (\partial \Delta_{(d_1,\ldots,d_n)} \wr \Delta) = \sum_{i=1}^n d_i +f_0(\Delta).$

Recall that $\dim (\partial \Delta_{(d_1,\ldots,d_n)} \wr \Delta) =\sum_{i=1}^n d_i+\dim(\Delta)$ where $\dim (\Delta) =d.$  The facets of the mixed wreath product are obtained from facets of $\Delta$ and each facet of the mixed wreath product are of the following form 
$$\bigcup_{v_i\in F} \{v_{i,0}, \ldots, v_{i,d_i}\} \cup \bigcup_{v_j \notin F} \{ v_{j,0}, \ldots, \widehat{v_{j,l}}, \ldots, v_{j,d_j}\}.$$ 
Let $f_{d} (\Delta)=k$  and $F_1, \ldots, F_k$ be the $d$-dimensional facets of $\Delta.$  For $v_j \in F_i$, we have the full $d_j$-simplex  $\Delta_{d_j}$ and for $v_j \notin F_i,$ we have each facet of boundary of a full $d_j$-simplex $\Delta_{d_j}$ contributing to a facet of the mixed wreath product. Let $F_i^c$ be the complement of $F_i$ where $F_i^c =  \{ v_{i,1}, \ldots, v_{i,n-d-1}\}.$ Then, for each $F_i,$ there are $\prod_{j=1}^{n-d-1} (d_{i,j}+1) $ facets of the mixed wreath product since $\partial \Delta_{d_j}$ has $d_j+1$ facets. Therefore, we have in total $\sum_{i=1}^k \prod_{j=1}^{n-d-1} (d_{i,j}+1) $ many facets of top dimension in the mixed wreath product.

One can also obtain a formula for each $f_i (\partial \Delta_{(d_1,\ldots,d_n)} \wr \Delta)$ when $1 < i <  \dim (\partial \Delta_{(d_1,\ldots,d_n)} \wr \Delta)$ in terms of the $f$-vectors of $\Delta.$ However, computations for the intermediate  $f$-vectors is more complex and tedious, therefore, we do not provide the explicit descriptions of the intermediate $f$-vectors for the sake of abbreviation.  
\end{proof}

In the remainder of this subsection, we focus on combinatorial properties of mixed wreath products such as vertex-decomposability, shellability, constructibility.


\begin{definition} Let $\Delta$ be a $d$-dimensional simplicial complex.
\begin{itemize}
\item A simplicial complex $\Delta$ is recursively defined to be \emph{vertex decomposable} if it is either a simplex or else has some vertex $v$ so that
\begin{enumerate}
\item both $\Delta \setminus \{v\}$ and $\link_{\Delta} (v)$ are vertex decomposable, and
\item no face of $\link_{\Delta}(v)$ is a facet of $\Delta \setminus \{v\}.$
\end{enumerate}
\end{itemize}
Vertex decompositions were introduced in the pure case by Provan and Billera in \cite{provan1980decompositions} and  extended to non-pure complexes by Bj\"orner and Wachs \cite[Section 11]{bjorner1997shellable} where they also defined shellable simplicial complexes in the non-pure set up.
 \begin{itemize}
\item Shellability is a way of putting together a simplicial complex while its faces fit together nicely. More precisely, a simplicial complex $\Delta$ is \emph{shellable} if its facets can be ordered $F_1, \ldots, F_t$
such that for all $1 \leq i<j \leq t,$ there exists some $v \in F_j \setminus F_i$ and some $l \in \{ 1, \ldots, j-1\}$ with $F_j \setminus F_l =\{v\}.$
\end{itemize}

 \begin{itemize}
\item $\Delta$ is \emph{constructible} if either $\Delta$ is a simplex or there are two $d$-dimensional subcomplexes
 $\Delta_1$ and $\Delta_2$ of $\Delta$ such that their union is $\Delta$ and $\Delta_1 \cap \Delta_2$ is  constructible $(d-1)$-dimensional simplicial complex. 
\end{itemize}
\begin{itemize}
\item (Reisner's Criterion) $\Delta$ is \emph{Cohen-Macaulay} if $\widetilde{H}_i (\link_{\Delta} (F);k) =0$ for all faces $F$ and $i < \dim (\link_{\Delta} (F))$ where $k$ is any field.  It is a well-known fact that any Cohen-Macaulay complex is pure.  
\end{itemize}
%
For pure simplicial complexes the following implications are strict (see \cite{bjorner1996topological}). 
$$\text{vertex decomposable } \implies \text{ shellable }  \implies \text{ constructible }  \implies \text{Cohen-Macaulay }$$

\end{definition}

If one drops the pureness condition and replaces Cohen-Macaulay with sequentially Cohen-Macaulay, above implications still hold  (see \cite{bjorner1996topological}). 

In \cite{joswig2005one}, it was shown that vertex-decomposability, shellability, and constructibility properties are preserved under the wreath product construction. Below, we generalize their results  \cite[Corollaries 4.2,  4.5, and 4.7]{joswig2005one} to mixed wreath products.

\begin{corollary}\label{cor:properties}
Let $\mathcal{P}$ be one of the following properties of a simplicial complex. 
\begin{enumerate}
\item vertex-decomposable,
\item shellable, or
\item constructible.
\end{enumerate}
The mixed wreath product of $\Delta$ is $\mathcal{P}$ if and only $\Delta$  is $\mathcal{P}.$ 
\end{corollary}

\begin{proof}
Recall that the mixed wreath product of a simplicial complex $\Delta$ is obtained by iterative applications of one-point suspensions. Let $\Delta (v)$ be a one-point suspension of $\Delta$ with respect to $v \in \Delta.$ Since $\Delta(v)$ is 
\begin{itemize}
\item vertex-decomposable if and only if  $\Delta$ is vertex decomposable by \cite[Proposition 2.5]{provan1980decompositions},
\item shellable   if and only if   $\Delta$ is shellable by \cite[Proposition 4.3]{joswig2005one}, 
\item constructible  if and only if  $\Delta$ is constructible by \cite[Proposition 4.6]{joswig2005one},
\end{itemize}
the mixed wreath product of $\Delta$ preserves the above properties  if and only if  $\Delta$ satisfies those properties. 
\end{proof}

We wish to point out that this list of properties can be extended and it is far from being an exhaustive list of properties which are preserved under the mixed wreath product construction.

\begin{remark}\label{rem:CM}
It is due to Munkres \cite{munkres1984topological} that Cohen-Macalayness  over a field is a topological property and the same work implies that one-point suspension of a Cohen-Macaulay simplicial complex is again a Cohen-Macaulay (also discussed in \cite{joswig2005one}). Since $\Delta$ appears as a link in $\Delta(v),$ then $\Delta$ must be Cohen-Macaulay for $\Delta(v)$  to be Cohen-Macaulay.  Thus,  Cohen-Macaulayness is a topological property shared by both $\Delta$ and its mixed wreath products. 
\end{remark}

\begin{remark}\label{rem:topprop}
It can also be shown that the mixed wreath product preserves topological properties such as being a cone, non-evasiveness, collapsibility, contractibility, $\mathbb{Z}$-acyclicity. (see \cite[Section 4.1]{joswig2005one}).
\end{remark}

\section{Weighted Simplicial Complexes}\label{Polar}

In this section, we focus on weighted simplicial complexes and propose a new approach to study them by introducing the  \emph{polarization} operation. Application of polarization to a weighted simplicial complex results with a mixed wreath product of the underlying simplicial complex.

\begin{definition}
Let $\Delta$ be a simpicial complex on the vertex set $V=\{x_1,\ldots, x_n\}$ and let $w: V \rightarrow \NN$ be a  function called \emph{weight function} which assigns a weight to each vertex of $\Delta.$ Let $w_i:= w(x_i)$ for each $i \in [n].$   

A \emph{vertex-weighted simpicial complex} is a pair $(\Delta,w)$ consisting of a simplicial complex $\Delta$ and a weight function $w$ on the set of vertices of $\Delta.$   For the remainder of the paper, we drop the term vertex and we read $(\Delta,w)$ as a weighted simplicial complex. 
\end{definition}

A weighted analog of Stanley-Reisner correspondence can be established by defining Stanley-Reisner ideals for weighted simplicial complexes.

\begin{definition}
Let $(\Delta,w)$ be a weighted simplicial complex on the vertex set  $V=\{x_1,\ldots, x_n\}$ with a weight function $w.$
The \emph{Stanley-Reisner ideal} of $(\Delta,w)$ is the ideal $ I_{(\Delta,w)}$ of $R=k[x_1,\ldots, x_n]$ which is generated by those monomials $\textbf{x}_F$ with $F \notin \Delta,$ in other words, 
$$I_{(\Delta,w)}= (\textbf{x}_F ~:~ F \in \mathcal{N} (\Delta))$$
where  $\displaystyle \textbf{x}_F:= \prod_{x_i \in F} x_i^{w_i}.$
\end{definition}

\begin{notation}
We use $x_1, x_2, \ldots$ to denote both vertices of a (weighted) simplicial complex and variables of a polynomial ring $R$ in this section.  For simplicity, we use $a, b, c, \ldots$ for variables and vertices in the examples.
\end{notation}

 \begin{example}

\begin{figure}[h]
\includegraphics[width=0.5\textwidth]{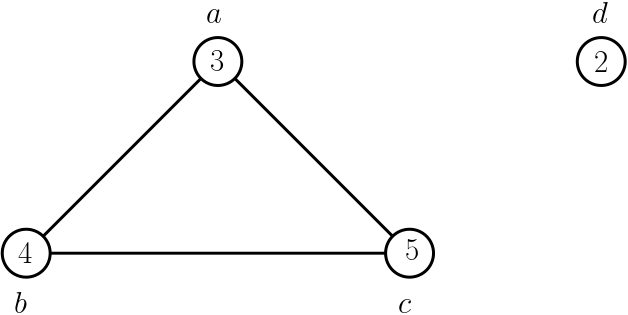}
\caption{Weighted simplicial complex}
  \label{fig:1}
\end{figure}
Let $(\Delta,w)$ be the weighted simplicial complex in Figure \ref{fig:1} on the vertices $\{a,b,c,d\}$ with weights $(3,4,5,2),$ in order. The Stanley-Reisner ideal of  $(\Delta,w)$   in $k[a,b,c,d]$ is
 $$ I_{(\Delta,w)} = ( a^3b^4c^5, a^3d^2, b^4d^2,c^5d^2 ).$$
\end{example}

\subsection{Polarization of weighted simpicial complexes} 

Polarization is a technique first used by Hartshorne in \cite{hartshorne1966} and, following his work, it became a popular tool in the study of monomial ideals. Polarization construction for weighted simplicial complexes borrows  ideas from the polarization of monomial ideals and it is defined as follows.

\begin{definition}\label{def:pol}
Let  $(\Delta,w)$ be a weighted simplicial complex on $n$ vertices. The \emph{polarization} of $(\Delta,w)$,  denoted by $(\Delta,w)^{\pol},$ is a simplicial complex on the new vertex set $V^{\pol}= \{x_{i,j} ~:~ 1 \leq i \leq n, 1 \leq j \leq w_i\}$ where $w_i$ is the weight of $x_i$ for each $i \in [n].$
\begin{itemize}
\item The minimal non-faces of $(\Delta,w)^{\pol}$  are all those subsets of the form 
$$ \bigcup_{x_i \in  F} \{x_{i,1}, \ldots, x_{i,w_i}\} $$
 where $F$ is a minimal non-face of $\Delta.$ 
\item The facets of $(\Delta,w)^{\pol}$ are all those subsets of $V^{\pol}$ of the form
$$ \bigcup_{x_i \in F} \{x_{i,1}, \ldots, x_{i,w_i}\} \cup \bigcup_{x_j \notin F} \{x_{j,1}, \ldots,  \widehat{x_{j,k}} , \ldots, x_{j,w_j}\}$$
where $F $ is a facet of $\Delta.$
\end{itemize}
\end{definition}

\begin{fact}
Let $(\Delta,w)$ be weighted simplicial complex on $n$ vertices.  The polarization of $(\Delta,w)$ is the mixed wreath product of $\Delta.$ More specifically, 
$$(\Delta,w)^{\pol} = \partial \Delta_{(w_1-1,\ldots, w_n-1)} \wr  \Delta.$$
\end{fact}

\begin{example}
\begin{figure}[h]
    \centering
    \begin{minipage}{.5\textwidth}
        \centering
\includegraphics[width=0.55\textwidth]{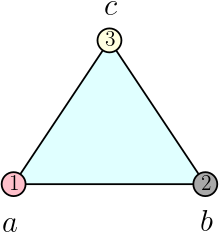}
\caption{$(\Delta,w)$}
  \label{fig:4}
    \end{minipage}%
    \begin{minipage}{0.5\textwidth}
        \centering
\includegraphics[width=0.75\textwidth]{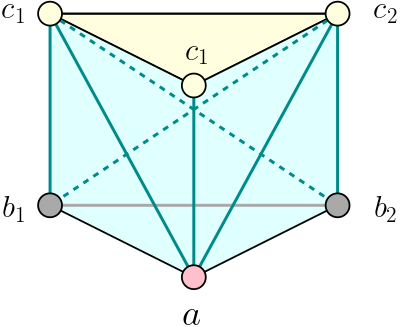}
\caption{$(\Delta,w)^{\pol}$ }
  \label{fig:5}
    \end{minipage}
\end{figure}
In \Cref{fig:4}, given weighted simlicial complex is a full triangle with the weights $(1,2,3).$ Then its polarization is a full 5-simplex displayed in  Figure \ref{fig:5}. In the polarization complex, we use the color of a vertex to denote the full simplex replacing that vertex. 
\end{example}

\begin{example}
 
\begin{figure}[h]
    \centering
    \begin{minipage}{.5\textwidth}
        \centering
\includegraphics[width=0.55\textwidth]{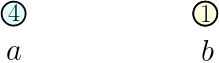}
\caption{$(\Delta,w)$}
  \label{fig:6}
    \end{minipage}%
    \begin{minipage}{0.5\textwidth}
        \centering
\includegraphics[width=0.9\textwidth]{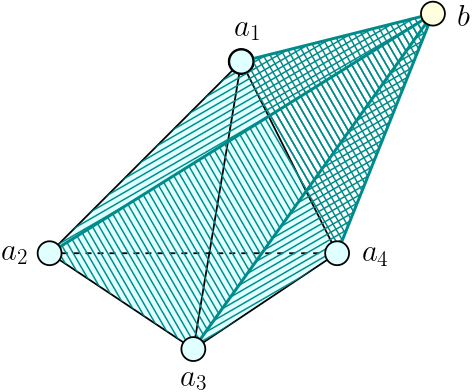}
\caption{$(\Delta,w)^{\pol}$}
  \label{fig:nice}
    \end{minipage}
\end{figure}

Let $(\Delta,w)$ the weighted simplicial complex given in Figure \ref{fig:6}.

 The only minimal non-face of the polarization $(\Delta,w)^{\pol}$  is $ \{a_1,a_2,a_3,a_4,b\}.$ Thus $(\Delta,w)^{\pol}$ is the boundary of  a pentachoron, full 4-simplex, given in \Cref{fig:nice}.

Note that $(I_{(\Delta,w)})^{\pol} = (a^4b)^{\pol}= (a_1a_2a_3a_4b)  = I_{(\Delta,w)^{\pol}}.$

\end{example}

Compared to the literature on the combinatorial and topological properties of simplicial complexes, we have very little information regarding the nature of weighted simplicial complexes. By using the established literature on simplicial complexes, we can view weighted simplicial complexes as combinatorial and topological objects through their underlying simplicial complexes and their polarizations.  

\begin{definition}\label{def:newCT}
Let $(\Delta,w)$ be a weighted simplicial complex and $\mathcal{P}$ be a combinatorial or a topological property. We say  $(\Delta,w)$ satisfies property $\mathcal{P}$ if and only if both $\Delta$ and $(\Delta,w)^{\pol}$ satisfy  $\mathcal{P}.$ 
\end{definition}

\begin{remark} 
A weighted simplicial complex satisfies all the properties which are preserved under the mixed wreath product construction. In particular, as a result of \Cref{cor:properties}, \Cref{rem:CM} and \Cref{rem:topprop}, one can take property $\mathcal{P}$ in \Cref{def:newCT}  to be  vertex-decomposability, shellability, or constructibility, non-evasiveness, collapsibility among other properties.
\end{remark}

Next we consider the relationship between monomial polarization operation and polarization of a weighted simplicial complex. There is a natural connection between the two polarizations through Stanley-Reisner ideals.
\begin{corollary}\label{cor:pol1}
Let $(\Delta,w)$ be a weighted simplicial complex on vertices $\{x_1,\ldots. x_n\}$ with its polarization $(\Delta,w)^{\pol}$ where $R=k[x_1,\ldots,x_n]$ and $R^{\pol}=k[x_{i,j} ~:~ 1 \leq i \leq n, 1 \leq j \leq w_i].$ Then
$$I_{(\Delta,w)^{\pol}} = (I_{(\Delta,w)})^{\pol}.$$
\end{corollary}

\begin{proof}
It follows immediately from the definitions of polarization of a monomial ideal and a weighted simplicial complex.
\end{proof}

A monomial ideal  and its polarization share many homological and algebraic properties. Thus, as an immediate consequence of \Cref{cor:pol} and \cite[Corollary 1.6.3]{herzog2010monomial}, we have the following relations.

\begin{corollary}\label{cor:pol}
Let $(\Delta,w)$ be a weighted simplicial complex on vertices $\{x_1,\ldots. x_n\}$ with its polarization $(\Delta,w)^{\pol}$ where $R=k[x_1,\ldots,x_n]$ and $R^{\pol}=k[x_{i,j} ~:~ 1 \leq i \leq n, 1 \leq j \leq w_i].$ Then
\begin{enumerate}
\item[(a)] $\beta_{i,j} (I_{(\Delta,w)}) = \beta_{i,j} (I_{(\Delta,w)^{\pol}})$ for all $i$ and $j.$
\item[(b)] $Hilb (R/ I_{(\Delta,w)},t ) = (1-t)^{\rho} Hilb (S/ I_{(\Delta,w)^{\pol}},t)$ where $\rho = (\sum_{i=1}^n w_i )-n.$
\item[(c)] $\height I_{(\Delta,w)} = \height I_{(\Delta,w)^{\pol}}.$
\item[(d)] $\pd I_{(\Delta,w)} = \pd I_{(\Delta,w)^{\pol}}$ and $\reg I_{(\Delta,w)} = \reg I_{(\Delta,w)^{\pol}}.$
\item[(e)] $R/ I_{(\Delta,w)}$ is Cohen-Macaulay (resp. Gorenstein) if and only if $R^{\pol}/ I_{(\Delta,w)^{\pol}}$ is Cohen-Macaulay (resp. Gorenstein).
\end{enumerate}
\end{corollary}

\begin{remark}
It follows from part (e) of the \Cref{cor:pol} and \Cref{rem:CM} that $R/I_{(\Delta,w)}$ is Cohen-Macaulay if and only if $R/I_{\Delta}$ is Cohen-Macaulay.
\end{remark}

We conclude this chapter with the following question which serves as the main motivation behind the next  chapter.

\begin{question}
What other algebraic relations can be established between $R/I_{(\Delta,w)}$ and $R/I_{\Delta}$?
\end{question}

\section{Weighted Monomial Ideals}\label{weightedmonomial}

In the previous chapter, the weight function defined on the set of vertices enabled us to obtain a non-squarefree monomial ideal from a given squarefree monomial ideal, namely, the Stanley-Reisner ideal of a simplicial complex. One can generalize this procedure to any monomial ideal and obtain a highly ``symmetrical" monomial ideal. We achieve this by defining a weight function on the set of variables of  a polynomial ring. 

\begin{definition}
Let $R=k[x_1,\ldots, x_n]$ be a polynomial ring and $w : \{x_1,\ldots, x_n\} \rightarrow \NN$  be a function. We call  $w$  a \emph{weight function} on the set of variables of $R$ and use $w_i$ to denote the value of the function $w$ at variable $x_i.$  Throughout the paper, we read the notation $w_i$ as the \emph{weight} of $x_i.$
\end{definition}

\begin{definition}\label{def:wideal}
Let $I \subseteq R=k[x_1,\ldots, x_n]$ be a monomial ideal and $w$ be a weight function on $\{x_1,\ldots, x_n\}.$ The \emph{weighted ideal} of $I$ with respect to $w$ is denoted by $(I,w)$ and defined as 
$$(I,w) =(x^{w(\textbf{b})} ~:~ x^{\textbf{b}} \in \G(I) )$$
where $w (\textbf{b}) = (w_1b_1, \ldots, w_nb_n)$ for an exponent vector $\textbf{b} = (b_1,\ldots, b_n) \in \NN^n.$
\end{definition}

\begin{example}
Let $I=(x_1x_2,x_2x_3x_4,x_5)  \subseteq k[x_1,\ldots, x_5]$ and let $(2,3,4,5,2)$ be the weights of the variables, in order. Then weighted ideal of $I$ is
$$(I,w)=(x_1^2x_2^3,x_2^3x_3^4x_4^5,x_5^2).$$
\end{example}

Weight functions defined on vertices of a simplicial complex and variables of a polynomial ring coincide by identifying vertices  $\{x_1,\ldots, x_n\}$ of a simplicial complex $\Delta$ with variables of a polynomial ring $R=k[x_1,\ldots, x_n].$  As a result of this identification, notions of the Stanley-Reisner ideal  of $(\Delta,w)$ agree with the weighted ideal of the Stanley-Reisner ideal of $\Delta.$ 

\begin{corollary}
Let $(\Delta,w)$ be a weighted simplicial complex with weight function $w$ on the vertices of simplicial complex $\Delta.$ Then 
$$I_{(\Delta,w)} = (I_{\Delta},w).$$
\end{corollary}

Due to this connection, we can establish a weighted version of the one-to-one correspondence between squarefree monomial ideals and simplicial complexes. 

\begin{center}
\begin{tikzcd}
\{ \text{simplicial  complexes}  \}  \arrow[r]   \arrow[d ]
        &  \{ \text{squarefree monomial ideals} \}    \arrow[l]   \arrow[d]\\
\{ \text{weighted simplicial  complexes}  \}    \arrow[r]   \arrow[u]   
&  \{ \text{weighted squarefree monomial ideals}\} \arrow[u]     \arrow[l]  
\end{tikzcd}
\end{center}

\subsection{Betti numbers of monomial ideals and their weighted ideals}

In the study of algebraic invariants of any monomial ideal $I$,  upper Koszul simplicial complexes introduced  in \cite{bayer1999}  are commonly used objects to investigate Betti numbers of $I.$  We recall the definition of this simplicial complex below.

\begin{definition}\label{def:koszul}
Let $I \subseteq R=k[x_1,\ldots, x_n]$ be a monomial ideal and let $\textbf{b}=(b_1, \ldots, b_n) \in \mathbb{N}^n$ be a $\mathbb{N}^n$-graded degree. The \emph{upper-Koszul simplicial complex} associated to $I$ at degree $\textbf{b},$ denoted by $K^{\textbf{b}}(I),$ is the simplicial complex over $V=\{x_1, \ldots, x_n\}$ whose faces are
$$\Big\{ W \subseteq V ~ | ~  x^{\textbf{b} - \tau } \in I \Big\}$$
where $x^{\textbf{b}} = x_1^{b_1} \cdots x_n^{b_n} $ and  $\displaystyle x^{\tau} = x_1^{\tau_1} \cdots x_n^{\tau_n}=\prod_{x_i \in W} x_i $ for a squarefree vector $\tau \in \{0,1\}^n.$
\end{definition}

One can compute the multigraded Betti numbers of a monomial ideal $I\subseteq R$ by using the following version of Hochster's formula  from \cite{bayer1999}.
\begin{eqnarray}\label{eq:koszul}
\beta_{i, \textbf{b}} (I) = \dim_k \widetilde{H}_{i-1} (K^{\textbf{b}}(I);k) \text{ for } i \ge 0 \text{ and } \textbf{b} \in \mathbb{N}^n.
\end{eqnarray}

Note that graded and multigraded Betti numbers are related, i.e., $\displaystyle \beta_{i,j} (I)= \sum_{\textbf{b} \in \mathbb{N}^n, |\textbf{b}| =j } \beta_{i, \textbf{b}} (I). $

\begin{remark}
Let $I \subseteq R=k[x_1,\ldots, x_n]$ be a monomial ideal and $\textbf{b} \in \mathbb{N}^n.$
 Suppose that $x^{\textbf{b}}$ is not the least common multiple of some subset of the minimal monomial generators of $I.$ Then $K^{\textbf{b}}(I)$ is a cone over some subcomplex. Thus all non-zero Betti numbers occur in $\mathbb{N}^n$-graded degrees $\textbf{b}$ for which $x^{\textbf{b}}$  equals at least a common multiple of some minimal generators.  In the case of  squarefree monomial ideals  all non-zero Betti numbers lie in squarefree degrees. 
\end{remark}

We first investigate the connection between upper Koszul complexes of a monomial ideal and its weighted ideal. It turns out that they are equal when degrees are adjusted according to weight values.
\begin{lemma}
Let $I \subseteq R=k[x_1,\ldots, x_n]$ be a monomial ideal and $(I,w)$ be its weighted ideal in $R.$
For any $\textbf{b} \in \mathbb{N}^n,$ we have
$$K^{\textbf{b}} (I) = K^{w(\textbf{b})} ((I,w))$$
where $w(\textbf{b}) = (w_1b_1, \ldots, w_nb_n).$
\end{lemma}

\begin{proof}
Let $W \subseteq \{x_1,\ldots, x_n\}$ be a face in $ K^{\textbf{b}} (I)$ with $ x^{\tau} = \prod_{x_i \in W} x_i $ such that $x^{\textbf{b}-\tau} \in I.$ Then there exists $x^{\textbf{a}} \in \G(I) $  dividing $x^{\textbf{b}-\tau},$ in particular, $x_i^{a_i}$ divides $x_i^{b_i-\tau_i}$ for each $i \in [n].$  Thus $x_i^{w_ia_i}$ divides $x_i^{w_ib_i-w_i \tau_i}$ for each $i \in [n]$ and  we conclude that $x^{w(\textbf{b})-w (\tau)}$  is divisible by $x^{w(\textbf{a})}.$ Therefore, $x^{w(\textbf{a})} \in \G((I,w))$  divides  $x^{w(\textbf{b})- \tau}$ and  $W$ must be a face in  $\in K^{w(\textbf{b})} ((I,w)).$

To prove the reverse inclusion, let $W$ be a face in $ K^{w(\textbf{b})} ((I,w))$ with   $ x^{\tau} = \prod_{x_i \in W} x_i$ for $\tau \in\{0,1\}^n.$ Then there exists $x^{w(\textbf{a})} \in \G((I,w))$ where $x^\textbf{a} \in \G(I)$ such that $x^{w(\textbf{a})}$ divides $ x^{w(\textbf{b})- \tau}$. This implies that $x_i^{w_ia_i} $ divides $x_i^{w_ib_i -\tau_i}$ which is equivalent to $w_ia_i \leq w_ib_i -\tau_i  $ for each $ i \in [n].$ Since $\tau \in \{0,1\}^n,$ each $\tau_i$ is either 0 or 1. If $\tau_i =0,$ then $a_i \leq b_i -\tau_i$ and if $\tau_i =1,$ we have $a_i  < b_i$ which is equivalent to $a_i  \leq  b_i -\tau_i.$ Thus $a_i \leq b_i -\tau_i$ for all $i$ and   $x^\textbf{a}$ must divide $x^{\textbf{b} -\tau}.$ Hence, $W$ is a face in $ K^{\textbf{\textbf{b}}} (I).$
\end{proof}

As a corollary of the previous lemma, we can establish the following relationship between the Betti numbers of an ideal and its weighted ideal by using  \eqref{eq:koszul}. 

\begin{corollary}\label{cor:betti}
Let $I \subseteq R=k[x_1,\ldots, x_n]$ be a monomial ideal and $(I,w)$ be its weighted ideal in $R.$ Then 
$$\beta_{i, (j_1, \ldots, j_n)} (I) = \beta_{i,(w_1j_1, \ldots, w_n j_n)} ((I,w))$$
for each $(j_1, \ldots, j_n) \in \mathbb{N}^n.$
\end{corollary}

Above relationship between the Betti numbers of $I$ and $(I,w)$ indicates equivalence of certain algebraic invariants and properties of a monomial ideal and its weighted ideal.

\begin{corollary}\label{cor:equalities}
Let $I \subseteq R=k[x_1,\ldots, x_n]$ be a monomial ideal and $(I,w)$ be its weighted ideal in $R.$ Then 
\begin{enumerate}
\item[(a)] $\pd ((I,w)) = \pd (I) \text{ and } \depth ((I,w)) = \depth (I).$
\item[(b)]  $R/(I,w)$ is Cohen-Macaulay if and only if $R/I$ is Cohen-Macaulay. 
\end{enumerate}
\end{corollary}

\begin{proof}
Equality of projective dimensions follows from Corollary \ref{cor:betti}. Since the projective dimensions are the same, we obtain the equality of depths from the Auslander-Buchsbaum formula.

Ideals $I$ and $(I,w)$ have the same radical ideal, thus their heights are equal and  $\dim (R/(I,w)) = \dim (R/I).$ Since $ \depth ((I,w)) = \depth (I)$, Cohen-Macaulayness is preserved under the weighting operation.
\end{proof}

\subsection{Associated primes of monomial ideals and their weighted ideals} In this subsection, we investigate the relationship between associated primes and primary decomposition of monomial ideals and their weighted ideals.

A primary decomposition of an ideal can be thought as an analogue of the unique factorization of integers to Noetherian rings. Primary decompositions of an ideal carries several important information about the ideal such as its associated primes, minimal primes, embedded primes, symbolic powers.  However, finding a primary decomposition for an ideal is a very difficult task, indeed, it is an  NP complete problem (see \cite{hocsten2002monomial}).  In the case of monomial ideals, structure of primary decompositions is well understood. 

\begin{definition}
Let $R$ be a Noetherian ring and $M$ be a finitely generated $R$-module.  A prime ideal $\mathfrak{p} $ in $R$ is called an \emph{associated prime ideal} of $M,$  if there exists an element $x \in M$ such that $ \mathfrak{p} = \Ann (x)$ where $\Ann(x)$ denotes the annihilator of  $x,$ i.e., $\Ann(x) = \{ r \in R ~:~ rx=0\}.$ The set of associated prime ideals of $M$ is denoted by $\Ass(M)$ (see \cite{Herzog2011} for definitions of primary decomposition, minimal prime, embedded prime and the remaining related terminology). 
\end{definition}

We recall the following fundamental fact about irredundant primary decomposition of monomial ideals.

\begin{theorem}\cite[Theorem 1.3.1, and Proposition 1.3.7]{Herzog2011}\label{thm:mondec}
Let $I \subset R=k[x_1,\ldots, x_n]$ be a monomial ideal. Then $I = \bigcap_{i=1}^k Q_i,$ where each $Q_i$ is generated by pure powers of the variables, is an irredundant primary decomposition of $I$. In other words, each $Q_i$ is of the form $(x_{i_1}^{a_1}, \ldots, x_{i_r}^{a_r})$ and $P$-primary component of $I$ is the intersection of all $Q_i$ with $\Ass(Q_i)=P$ where $P= (x_{i_1}, \ldots, x_{i_r}).$ Moreover, an irredundant primary decomposition of this form is unique.
\end{theorem}

One of the algorithms to obtain the irredundant primary decomposition of a monomial ideal is given in \cite[Theorem 1.3.1]{Herzog2011}.  In the following example, we demonstrate this algorithm.

\begin{example}
Let $I=(x_1^2x_2^3,x_2^3x_3^4,x_3^4x_1^2)\subseteq k[x_1,x_2,x_3].$ First step of the algorithm starts by choosing a minimal monomial generator that is not generated by pure powers of a subset of variables. Thus we may choose any generator  of $I$ in the first step. Next steps of the decomposition follows in the same fashion.
\begin{eqnarray*}
I&=& (x_1^2x_2^3,x_2^3x_3^4,x_3^4x_1^2) \\
&=& (x_1^2,x_2^3x_3^4,x_3^4x_1^2)\cap (x_2^3,x_2^3x_3^4,x_3^4x_1^2) =(x_1^2,x_2^3x_3^4)\cap (x_2^3,x_3^4x_1^2)   \\
&=& (x_1^2,x_2^3) \cap  (x_1^2,x_3^4 ) \cap (x_2^3,x_3^4) \cap (x_2^3,x_1^2)\\
&=& (x_1^2,x_2^3) \cap  (x_1^2,x_3^4 ) \cap (x_2^3,x_3^4) . \\
\end{eqnarray*} 
Let $J=(x_1x_2,x_2x_3,x_3x_1).$ Ideals $I$ and $J$ are related through the weighting operation with the weight function $w$ such that $w_1=2,w_2=3,w_3=4.$ Furthermore, $J$ has the following irredundant primary decomposition.
$$J= (x_1,x_2) \cap  (x_1,x_3 ) \cap (x_2,x_3)$$
Note that associated primes of $I$ and $J$ are equal. Moreover, irredundant primary decomposition of $I$ is obtained from that of $J$ through weighing operation with respect to $w.$
\end{example}

In what follows, we show that the concluding remarks of the above example are true in general.  We first prove that associated primes of a monomial ideal  $I$ and its weighted ideal coincide. 

\begin{theorem}\label{ref:assprimes}
Let $I \subseteq R=k[x_1,\ldots, x_n]$ be a monomial ideal and $(I,w)$ be its weighted ideal in $R.$  Then
$$ \Ass (R/(I,w)) = \Ass (R/ I).$$ 
\end{theorem}

\begin{proof}
One side of the inclusion $\Ass (R/ (I,w)) \subseteq \Ass (R/ I)$ is immediate from the fact that  $(I,w) \subseteq I.$ For the reverse inclusion, let  $\pp \in \Ass (R/ I).$ Then there exists a monomial $m \in R$ such that $\pp = I:m.$ Note that $\pp$ is generated by a subset of variables in $R=k[x_1,\ldots, x_n]$ and let $x_i$ be a variable in $\pp,$ set $i=1.$ Then $x_1m \in  I.$ Let  $ f \in \G(I)$ such that 
\begin{eqnarray*}
x_1m &=& f m' 
\end{eqnarray*}
 for a monomial $m' $ in  $R.$ 
 
Let $f= x^{\textbf{f}}, m=x^{\textbf{m}},$ and $m'=x^{\textbf{m'}}$ where $\textbf{f}= (f_1,\ldots, f_n),  \textbf{m}= (m_1,\ldots, m_n), \textbf{m}'= (m'_1, \ldots, m'_n) \in \NN^n.$ It is clear that $x_1$ divides $f$ and $m'$ is not divisible by $x_1.$  Otherwise, $m \in I,$ a contradiction. Thus, we have $f_1\geq 1$ and $m'_1=0.$  Then
 \begin{eqnarray*}
x_1^{m_1} \cdots x_n^{m_n} &=& x_1^{f_1-1}  x_2^{f_2+m'_2} \cdots x_n^{f_n+m'_n} 
\end{eqnarray*}
 where $m_1=f_1-1$ and $m_i= f_i+m'_i$ for each $i \in \{ 2,\ldots, n\}.$ 
 
 
Let $F= x_1^{w_1f_1} \ldots x_n^{w_nf_n}$ be the minimal generator of $(I,w)$ corresponding to $f\in \G(I)$ and set  $M= \prod_{i=1}^n x_{i}^{w_if_i-f_i}.$ It follows from the construction of $M$ that $ x_1mM = F m' \in (I,w).$

 We wish to show that $mM \notin (I,w).$  On the contrary, suppose $mM \in  (I,w).$ Then there exists $x^{w(\textbf{b})} \in \G((I,w))$ where $x^{\textbf{b}} \in \G(I)$ such that $x^{w(\textbf{b})} $  divides 
\begin{eqnarray*}
mM&=&x_1^{w_1f_1-1} x_2^{w_2f_2+m'_2}  \cdots x_n^{w_nf_n+ m'_n}.
 \end{eqnarray*} 

This means $w_1b_1 \leq w_1f_1-1 <w_1f_1 $ and $w_ib_i  \leq w_if_i+ m'_i \leq w_if_i+ w_im'_i $ for each $i \in\{2,\ldots, n\}.$ Since  $w_i \geq 1$ for each $i,$ we conclude that $  b_1\leq f_1 -1 =m_1$ and $ b_i \leq f_i+m'_i=m_i.$ It implies that $x^{\textbf{b}}$ divides $m$ and $m$ must be  in $I,$ a contradiction. Therefore, variable $x_1$ is  in $(I,w): mM$ for any  $x_1$ in $\pp.$ Thus $\pp \in \Ass (R/ (I,w)),$ proving the reverse inclusion.
\end{proof}

This result was also proved in \cite[Lemma 3.9]{sayedsadeghi2018normally} by constructing multiple auxiliary lemmas and considering primary decompositions as a starting point. In contrast, we first establish the equality of associated primes by solely using the structure of associated primes and the weighted ideal. In addition, our proof is considerably more compact. 

Using the equality of associated primes as a base, we can conclude results in regards to primary decompositions and normal torsion freeness.

\begin{corollary}\label{cor:primdecom}
Let $I \subseteq R=k[x_1,\ldots, x_n]$ be a monomial ideal and $(I,w)$ be its weighted ideal in $R.$ If $I= \bigcap_{i=1}^k Q_i$ is the irredundant primary decomposition of $I,$  then  $(I,w)= \bigcap_{i=1}^k (Q_i,w)$ is the irredundant primary decomposition of $(I,w).$ 
\end{corollary}

\begin{proof}
Let $I= \bigcap_{i=1}^k Q_i$ be the irredundant primary decomposition of $I,$  where $Q_i$ is a $P_i$-primary ideal for each $i$ and  $P_i \in \Ass(R/I).$  Note that $( (J \cap K) ,w )= (J,w) \cap (K,w)$ for any two monomial ideals in $R.$ Thus, we have  $(I,w)= \bigcap_{i=1}^k (Q_i,w)$ where each $(Q_i,w)$ is a $P_i$-primary ideal for each $i.$  Since $ \Ass (R/(I,w)) = \Ass (R/ I),$ the decomposition $(I,w)= \bigcap_{i=1}^k (Q_i,w)$ is the irredundant primary decomposition of $(I,w).$ 
\end{proof}

Another important notion of an ideal is normal torsion freeness which investigates associated primes of an ideal and its powers.

 \begin{definition}
An ideal  $I\subseteq R$ is called \emph{normally torsion free} if  $\Ass (R/I^k) \subseteq \Ass (R/I)$ for all $k \geq 1.$
 \end{definition}
 
One can immediately see from \Cref{ref:assprimes} that normally-torsion freeness is preserved under the weighting operation.
 
\begin{corollary}\label{cor:normaltorsion}
Let $I \subseteq R=k[x_1,\ldots, x_n]$ be a monomial ideal and $w$ be a weight function on the variables of $R.$  Then $I$ is normally torsion free if and only if $(I,w)$ is normally torsion free.
\end{corollary}

One of the interesting questions in commutative algebra is to determine when the ordinary and symbolic powers of a given ideal coincide. In general, this question is open. In some known cases, there are conditions on the ideal that are equivalent to the equality of the ordinary and symbolic powers. We recall one such condition below (see \cite{Herzog2011} for definition of symbolic power of an ideal).

\begin{proposition}\cite[Proposition 3.3.26]{villarreal2000monomial} \label{prop:power}
Let $I$ be an ideal of a ring $R.$ If $I$ has no embedded primes, then $I$ is normally torsion free if and only if $I^s=I^{(s)}$ for all $s \geq 1.$
\end{proposition}

\begin{remark}\label{rem:torsionfree}
It follows from \Cref{ref:assprimes} and \Cref{cor:primdecom} that embedded primes of $I$ and $(I,w)$ coincide. Thus, when $I$ has no embedded primes, we obtain another class of ideals, $(I,w)$ for any weight function $w,$ in which  equality of ordinary and symbolic powers  is equivalent to normal torsion freeness. 
\end{remark}

In the class of edge ideals, there is a nice combinatorial characterization (\cite[Theorem 5.9]{SIMIS1994})  for normal torsion freeness of edge ideals (i.e., equality of ordinary and symbolic powers). In the following example, we consider edge ideals of vertex-weighted graphs and accomplish the same characterization for this class of ideals. 

\begin{example}
Let $G=(V,E)$ be a graph with the vertex set $V=\{x_1,\ldots, x_n\}$ and edge set $E.$ Edge ideal of $G$ in $R=k[x_1,\ldots, x_n]$ is denoted by $I(G)$ and it is defined as
$$I(G) = (x_ix_j  ~:~ \{x_i,x_j\} \in E).$$

It is well-known that associated primes of $I(G)$ are related to minimal vertex covers of $G.$ In particular, let $\mathcal{C}$ be the set of minimal vertex covers of $G,$ then  $I(G) = \bigcap_{F \in \mathcal{C}} Q_{F}$  is the irredundant primary decomposition  of $I(G)$ where $Q_{F}$ is the primary ideal $(x_i ~:~ x_i \in  F).$ Moreover,  $I(G)$ has no embedded primes since $\Ass(R/I (G)) = \{ (x_i : x_i \in  F) :  F \in \mathcal{C}\}.$ 

Let $w$ be a weight function on the set of variables of $R$ which is also a weight function defined on the vertices of $V.$ 
Consider the vertex-weighted graph $G$ with the weight function $w$ and denote it by $(G,w).$ Similar to the unweighted case, one can define the edge ideal of a vertex-weighted graph as follows
$$I(G,w)= (x_i^{w_i}x_j^{w_j}  ~:~ \{x_i,x_j\} \in E).$$

It is clear that $I(G,w) = (I(G),w)$ and, if $w_i=1$ for each $i,$  we have $I(G,w)=I(G).$

As an application of  \Cref{ref:assprimes}, we see that associated primes of $I(G,w) $ come from the minimal vertex covers of $G$ and $I(G,w)$ has no embedded primes.

Normally torsion free  edge ideals were classified  in \cite[Theorem 5.9]{SIMIS1994} where it was shown that  $G$ is bipartite if and only if its edge ideal $I(G)$ is normally torsion free, equivalently, $I(G)^s$ has no embedded primes for any $s \geq 1.$ Thus, we obtain the following weighted analogue of \cite[Theorem 5.9]{SIMIS1994} as a result of \Cref{cor:normaltorsion}. 
\end{example}

\begin{center}
\begin{tikzcd}
  G \text{ is bipartite }  \arrow[r]   \arrow[d ]
    &  I(G) \text{ is normally torsion free }   \arrow[l]   \arrow[d]\\
  (G,w) \text{ is bipartite }   \arrow[u]   
&  I(G,w) \text{ is normally torsion free } \arrow[u]   
\end{tikzcd}
\end{center}

Therefore, by \Cref{prop:power}, we have $ I(G,w)^s =  I(G,w)^{(s)}$ for all $s \geq 1$ if and only if $G$ is bipartite. 

We conclude the paper with the following questions.

\begin{question}
As concluded in \Cref{cor:equalities}, projective dimension and depth of a monomial ideal and its weighted ideal are equal. Another important algebraic invariant, in general, is the Castelnuovo-Mumford regularity. What kind of relations can be established for the regularity of a monomial ideal and its weighted ideal?
\end{question}

\begin{question}
We have shown in \Cref{cor:properties} that vertex-decomposability, shellability, constructibility are preserved under the mixed wreath product construction. What other properties are preserved under this construction?
\end{question}

%
%

\bibliographystyle{plain}
\bibliography{references}

\end{document}